\documentclass[a4paper,12pt,reqno]{amsart}
\usepackage{amsfonts, color}
\usepackage{amsmath}
\usepackage{amssymb}
\usepackage[a4paper]{geometry}
\usepackage{mathrsfs}
\usepackage[colorlinks]{hyperref}
\usepackage{comment}
\usepackage[colorinlistoftodos]{todonotes}

\usepackage{hyperref}
\renewcommand\eqref[1]{(\ref{#1})}

\makeatletter
\newcommand*{\mint}[1]{%
  \mint@l{#1}{}%
}
\newcommand*{\mint@l}[2]{%
  \@ifnextchar\limits{%
    \mint@l{#1}%
  }{%
    \@ifnextchar\nolimits{%
      \mint@l{#1}%
    }{%
      \@ifnextchar\displaylimits{%
        \mint@l{#1}%
      }{%
        \mint@s{#2}{#1}%
      }%
    }%
  }%
}
\newcommand*{\mint@s}[2]{%
  \@ifnextchar_{%
    \mint@sub{#1}{#2}%
  }{%
    \@ifnextchar^{%
      \mint@sup{#1}{#2}%
    }{%
      \mint@{#1}{#2}{}{}%
    }%
  }%
}
\def\mint@sub#1#2_#3{%
  \@ifnextchar^{%
    \mint@sub@sup{#1}{#2}{#3}%
  }{%
    \mint@{#1}{#2}{#3}{}%
  }%
}
\def\mint@sup#1#2^#3{%
  \@ifnextchar_{%
    \mint@sup@sub{#1}{#2}{#3}%
  }{%
    \mint@{#1}{#2}{}{#3}%
  }%
}
\def\mint@sub@sup#1#2#3^#4{%
  \mint@{#1}{#2}{#3}{#4}%
}
\def\mint@sup@sub#1#2#3_#4{%
  \mint@{#1}{#2}{#4}{#3}%
}
\newcommand*{\mint@}[4]{%
  \mathop{}%
  \mkern-\thinmuskip
  \mathchoice{%
    \mint@@{#1}{#2}{#3}{#4}%
        \displaystyle\textstyle\scriptstyle
  }{%
    \mint@@{#1}{#2}{#3}{#4}%
        \textstyle\scriptstyle\scriptstyle
  }{%
    \mint@@{#1}{#2}{#3}{#4}%
        \scriptstyle\scriptscriptstyle\scriptscriptstyle
  }{%
    \mint@@{#1}{#2}{#3}{#4}%
        \scriptscriptstyle\scriptscriptstyle\scriptscriptstyle
  }%
  \mkern-\thinmuskip
  \int#1%
  \ifx\\#3\\\else_{#3}\fi
  \ifx\\#4\\\else^{#4}\fi
}
\newcommand*{\mint@@}[7]{%
  \begingroup
    \sbox0{$#5\int\m@th$}%
    \sbox2{$#5\int_{}\m@th$}%
    \dimen2=\wd0 %
    \let\mint@limits=#1\relax
    \ifx\mint@limits\relax
      \sbox4{$#5\int_{\kern1sp}^{\kern1sp}\m@th$}%
      \ifdim\wd4>\wd2 %
        \let\mint@limits=\nolimits
      \else
        \let\mint@limits=\limits
      \fi
    \fi
    \ifx\mint@limits\displaylimits
      \ifx#5\displaystyle
        \let\mint@limits=\limits
      \fi
    \fi
    \ifx\mint@limits\limits
      \sbox0{$#7#3\m@th$}%
      \sbox2{$#7#4\m@th$}%
      \ifdim\wd0>\dimen2 %
        \dimen2=\wd0 %
      \fi
      \ifdim\wd2>\dimen2 %
        \dimen2=\wd2 %
      \fi
    \fi
    \rlap{%
      $#5%
        \vcenter{%
          \hbox to\dimen2{%
            \hss
            $#6{#2}\m@th$%
            \hss
          }%
        }%
      $%
    }%
  \endgroup
}
\numberwithin{equation}{section}
\theoremstyle{plain}
\newtheorem{thm}{Theorem}[section]
\newtheorem{prop}[thm]{Proposition}
\newtheorem{cor}[thm]{Corollary}
\newtheorem{lem}[thm]{Lemma}

\theoremstyle{definition}

\newtheorem{rem}[thm]{Remark}

\newcommand{\ov}{\overline}

\newcommand{\cl}{{C \kern -0.1em \ell}}
\newcommand{\R}{\mathbb{R}}
\newcommand{\BR}{\mathbb{R}}

\newcommand{\cH}{\mathcal{H}}
\newcommand{\B}{B_{1}^{c}}

\newcommand{\llangle}{\langle \kern -0.2em \langle}	
\newcommand{\rrangle}{\rangle \kern -0.2em \rangle}

\newcommand{\inner}[1]{\left\langle  #1 \right\rangle } 
\newcommand{\rpart}[1]{\left[  #1 \right]_0 }

\title[Zero modes and  Dirac-(logarithmic) Sobolev-type inequalities]{Zero modes and Dirac-(logarithmic) Sobolev-type inequalities}
\author[M. Chatzakou]{Marianna Chatzakou}
\address{
	Marianna Chatzakou:
	\endgraf
    Department of Mathematics: Analysis, Logic and Discrete Mathematics
    \endgraf
    Ghent University, Belgium
  	\endgraf
	{\it E-mail address} {\rm marianna.chatzakou@ugent.be}
		}
	
\author[U. K\"ahler]{Uwe K\"ahler}
\address{
  Uwe Kaehler:
  \endgraf
 Center for research and development in mathematics and applications
 \endgraf
 Department of Mathematics
  \endgraf
  University of Aveiro
  \endgraf
  3810-193 Aveiro
  \endgraf
  Portugal
  \endgraf
	{\it E-mail address}  {\rm ukaehler@ua.pt}}

\author[M. Ruzhansky]{Michael Ruzhansky}
\address{
  Michael Ruzhansky:
  \endgraf
  Department of Mathematics: Analysis, Logic and Discrete Mathematics
  \endgraf
  Ghent University, Belgium
  \endgraf
 and
  \endgraf
  School of Mathematical Sciences
  \endgraf
  Queen Mary University of London
  \endgraf
  United Kingdom
  \endgraf
  {\it E-mail address} {\rm michael.ruzhansky@ugent.be}
  }

\begin{document}

\thanks{M. Chatzakou and M. Ruzhansky are supported by the FWO Odysseus 1 grant G.0H94.18N: Analysis and Partial Differential Equations and by the Methusalem programme of the Ghent University Special Research Fund (BOF) (Grant number 01M01021). M. Chatzakou is a postdoctoral fellow of
the Research Foundation-Flanders (FWO) under the postdoctoral grant No 12B1223N. U. K\"ahler was supported by a grant for a scientific stay in Flanders of FWO Flanders under the number
V500824N. U. K\"ahler was also (partially) supported by CIDMA and is funded by the Funda\c{c}\~ao para
a Ci\^encia e a Tecnologia, I.P. (FCT, Funder ID = 50110000187). M. Ruzhansky is also supported by EPSRC grant EP/V005529/1 and by FWO Senior Research Grant G022821N. }

\begin{abstract} 
We study the decay rate of the zero modes of the Dirac operator with a matrix-valued potential that is considered here without any regularity assumptions, compared to the existing literature. For the Dirac operator and for Clifford-valued functions we prove the $L^p$-$L^2$ Dirac Sobolev inequality with explicit constant, as well as the $L^p$-$L^q$ Dirac-Sobolev inequalities. We prove its logarithmic counterpart for $q=2$, extending it to its Gaussian version of Gross, as well as show Nash and  Poincar\'e inequalities in this setting, with explicit values for constants. 
\end{abstract}

\maketitle

\section{Introduction}

  The main motivation of this work is the study of the decay of the zero modes,  i.e., the eigenfunctions corresponding to the zero eigenvalue,   of the Dirac operator with a matrix-valued  potential, that is of an operator of the form
\begin{equation}\label{def:dirac}
D+Q=\sum_{j=1}^{n}e_j \partial_{x_j}+Q\,,
\end{equation}
where $\{e_1,\ldots,e_n \}$ stands for the standard basis of the Euclidean space $\R^n$,  $n\geq 1$.  Roughly speaking,  this problem is the mathematical analogue of the physical problem of the stability of matter.  This is a classical problem in physics for several decades, cf. \cite{LS10} for a recent overview of the topic, that  has been a major topic of study for many authors since the 60's.  While giving a complete list of related works would be impossible,  we can refer the interested reader to the works \cite{AMN00}, \cite{PSS18}, \cite{BES08} which offer an  overview of the topic, with an emphasis on the mathematical consideration of it.   

The zero modes for the above problem are, vaguely speaking, the lower estimates of the energy of a system of particles.  It is known,  see \cite{BE02},  that in the classical case,  zero modes are rare   and their decay rate,  whenever they exist,  is well studied.  

In the current consideration, the classical results are studied in  the case where the differential operator is the Dirac operator as in \eqref{def:dirac}, while the functions that we consider here, and so also the zero modes, are Clifford-valued. Although in particle physics the functions to be considered are spinor-valued, it is much easier and common to work with Clifford-valued functions as a larger set. The spinor-valued results then follow as a special case. In the special case of Minkowski space-time the space of functions can even be enlarged to  the case of 2 by 2 complex matrix-valued functions. However, in the latter case, getting  a physical interpretation of the results is quite difficult. 

While the possibility of the existence of the zero mode is not within the scope of the current work,  here we manage to sharpen the results of \cite{BES08} by using different methods,  not only by showing higher decay rate of the zero modes, but also by removing the regularity assumptions on $Q$.  Zero modes in this setting when the regarded Dirac operator appears with a ``twist'', in which case it is also called the ``Pauli operator'',  have also been studied in \cite{BEU11}.  In the same work,  the authors also establish Dirac-Sobolev and Dirac-Hardy inequalities, but the analysis there is with respect to the weak $L^p$-spaces of functions. Furthermore, we would like to mention~\cite{DELV04,DECV07} where the authors studied  Hardy-type inequalities with respect to the operator (\ref{def:dirac}) aiming for min-max characterizations of the discrete spectrum.

The analysis of Dirac-Sobolev embeddings and of the Dirac-Sobolev inequalities has been the topic of several works; e.g. we refer to the work \cite{IS10} where  the authors  study the relationship of the $L^p$-Dirac-Sobolev spaces with the classical Sobolev spaces in an open domain in $\R^3$,  while in \cite{IS10, BR07} functional inequalities that involve the Dirac operator and Dirac-type operators have been proved.  Another work in this direction is \cite{ST13} where the authors prove the critical Dirac-Sobolev inequality,  giving also applications to isoperimetric problems.  Dirac-Sobolev inequalities,  but in settings different than the Euclidean one,  have also been considered,  cf. the work \cite{R09} on  closed compact Riemannian spin manifolds.  

Despite the amount of works around the  Dirac-Sobolev embeddings in the above sense,   the classical $L^p$-$L^2$ Dirac-Sobolev inequality with explicit constant for Clifford-valued function was still an open problem. Here we prove the latter inequality with an explicit constant, as well as the general $L^p$-$L^q$ Dirac-Sobolev inequalities in this setting. 

Having as a starting point the Dirac-Sobolev inequality,  we prove  other functional inequalities, for Clifford-valued functions and with respect to the Dirac operator $D$, that are also very important especially in view of their applications to partial differential equations.  These inequalities include: (i) the logarithmic-Sobolev inequality; (ii) the Nash inequality; (iii) the Gaussian  logarithmic-Sobolev inequality,  and (iv) the Poincar\'e inequality.

An overview of the structure of the paper is as follows: in Section \ref{SEC:pre} we give a short exposition of the notions of Clifford algebras and Clifford-Hilbert modules.  We recall some known classical functional inequalities in the setting of Clifford-valued functions,  and define the Dirac-Sobolev spaces in the above setting.  In Sections \ref{SEC:sob.ineq} and  \ref{SEC:log.SobandPoi} we prove the aforesaid inequalities.  In the last section,  Section \ref{SEC:zero},  we obtain $L^p$-estimates for the zero modes of the Dirac operator as in \eqref{def:dirac},  improving the results obtained in \cite{BES08}. 

Finally, let us mention that even though for the functional inequalities in Sections \ref{SEC:sob.ineq} and  \ref{SEC:log.SobandPoi} the results are stated for functions in $\mathbb{R}_{0,n}$ they are actually true for Clifford-valued functions in $\mathbb{C}_n$ after obvious modifications in the corresponding proofs. The range of applications of the results obtained here lies in the context of shallow-water waves (cf. \cite{Wu99}), Hodge-Dirac operator (cf. \cite{MM18}), or Atiyah-Singer-Dirac operators (cf. \cite{BMR18}).

\section{Preliminaries}\label{SEC:pre}

\subsection{Clifford algebras}
Let $\{ e_1, \cdots, e_n\}$ be the standard basis of the Euclidean vector space in $\BR^n$. The associated Clifford algebra $\BR_{0,n}$ is the free algebra generated by $\BR^n$ modulo $x^2 = - \sum_{i=1}^{n} x_i^2$. The defining relation induces the multiplication rules $e^2_i = -1$ for $i=1, \cdots, n,$ and $e_i e_j + e_j e_i = 0, i\not=j$.\smallskip

A vector space basis for $\BR_{0,n}$ is given by the set
\begin{equation} \label{Eq:01}
\left\{e_\emptyset=1, e_A  = e_{l_1} e_{l_2} \ldots e_{l_r}: ~A = \{ l_1, l_2, \ldots, l_r \}, 1 \leq l_1 < \ldots < l_r \leq n\right\}.
\end{equation}
Each $a \in \BR_{0,n}$ can be written in the form $a=\sum_A a_A ~e_A$, with $a_A \in \BR.$ Moreover, each element $a=\sum_A a_A ~e_A$ decomposes into $k$-blades $[a]_k := \sum_{A: \# A = k} a_A ~e_A$ with $a=\sum_{k=0}^n [a]_k,$ and we write $\BR_{0,n} = \cup_{k=0}^n ~ \BR_{0,n}^k,$ where $\BR_{0,n}^k := \{ [a]_k, a \in \BR_{0,n} \}.$\smallskip

We will also consider the complexified Clifford algebra $\mathbb{C}_n=\mathbb{R}_{0,n}\otimes\mathbb{C}$. This notation is due to the fact that a complex Clifford algebra has no signature, i.e. the algebras $\mathbb{R}_{0,n}\otimes\mathbb{C}$ and $\mathbb{R}_{1,n-1}\otimes\mathbb{C}$ are the same. 

The conjugation in the Clifford algebras $\BR_{0,n}$ and $\mathbb{C}_n$ is defined as the automorphism $x \mapsto \ov{a} =\sum_A  \ov{a_A} ~\ov{e}_A,$ where $\ov{e}_\emptyset =1,  \ov{e}_j =-e_j (j=1,\ldots,{n}),$ $\ov{e}_A =\ov{e}_{l_r} ~\ov{e}_{l_{r-1}} \ldots \ov{e}_{l_1}$, and $\overline{a_A}$ is the complex conjugate of $a_A$. In particular, we get $\overline{a}=(-1)^{k(k+1)/2}a$ for $a\in\mathbb{R}_{0,n}^k$. An important property is $e_A\overline{e_A}=\overline{e_A}e_A=1$ for all $A$.

A norm in the Clifford algebras $\BR_{0,n}$ and $\mathbb{C}_n$ is given by $|a|^2=\sum_A |a_A|^2$ for any $a\in\mathbb{R}_{0,n}$ or any $a\in\mathbb{C}_{n}$. Hereby, we have
$$
|a|^2=[a\overline{a}]_0 \mbox{ for all }a\in\BR_{0,n} \mbox{ or }a\in\mathbb{C}_{n}.
$$

For a vector $x=\sum_{j=1}^n x_j e_j \in \BR^n \mbox{ or }\mathbb{C}^n$ we have $x \overline x = | x|^2 :=\sum_{j=1}^n|x_j |^2.$ Hence, each non-zero vector $x=\sum_{j=1}^n x_j e_j$ with $|x|\neq 0$ has an unique multiplicative inverse given by $x^{-1} = \frac{\ov x}{| x|^2}=-\frac{x}{| x|^2}.$ Up to the minus sign this inverse corresponds to the Kelvin inverse of the vector and geometrically represents a reflection on the unit sphere in $\BR^n$. 

In what follows we will concentrate on the case of $\mathbb{R}_{0,n}$ but all results easily carry over into the case of $\mathbb{C}_n$.

Additionally, we have the inequality $|ab|\leq K_n |a||b|$ for $a,b\in\BR_{0,n}$ (\cite{Hile90}) where the sharp constant is given by
\begin{equation}
    \label{Kn}
    K_n=\left\{\begin{array}{ll} 2^{n/4} & \mbox{if $n=0,6 \mod 8$}\\
2^{(n-1)/4} & \mbox{if $n=1,3,5 \mod 8$}\\
2^{(n-2)/4} & \mbox{if $n=2,4 \mod 8$}\\
2^{(n+1)/4} & \mbox{if $n=7 \mod 8$}\,.\end{array}\right.
\end{equation}

Similarly, for the case of the complexified Clifford algebra $\mathbb{C}_n$ we have
\begin{equation}
    \label{Knc}
K_n=\left\{\begin{array}{ll} 2^{n/4} & n \mbox{ even }\\
2^{(n+1)/4} & n \mbox{ odd }\,.\end{array}\right.
\end{equation}
In the special case that either $a$ or $b$ is a vector, we have the inequality $|ab|\leq |a||b|$. For details, we refer to~\cite{Hile90}.

An $\BR_{0,n}$-valued function $f$ over a non-empty domain $\Omega \subset \BR^{n}$ is written as  $f=\sum_A f_A e_A$, with components $f_A:\Omega \rightarrow \BR$. Properties such as continuity are to be understood component-wisely. For example, the function $f$ as above is continuous if and only if all components $f_A$ are continuous. Finally, we recall the Dirac operator $$D=\sum_{j=1}^{n} e_j ~\partial_{x_j},$$ 
which factorizes the Laplacian, i.e., $D^2=- \Delta = -\sum_{j=1}^n \partial_{x_j}^2$. 

A $\BR_{0,n}$-valued function $f$ is said to be  {\it left-monogenic} if it satisfies $Df=0$ on $\Omega$ (resp. \textit{right-monogenic} if it satisfies $fD=0$ on $\Omega$).

For the special case of the algebra $\mathbb{C}_4$ in~\cite{BES08}, the complexified space-time algebra (often also somewhat misleadingly denoted by $\mathbb{R}_{1,3}\otimes\mathbb{C}$ in papers coming from physics), and in many other papers the Dirac operator is represented by
$$
D=\alpha\cdot (-i\nabla), 
$$
where $\alpha$ denotes the triple of Dirac matrices
$$
\alpha_j=\left(\begin{array}{cc}
0     &  \sigma_j \\
 \sigma_j & 0
\end{array}\right)
$$
with $\sigma_j$ being the Pauli matrices and $-i\nabla=\left(-i\partial_{x_1},-i\partial_{x_2}, -i\partial_{x_3}\right)$. While this is a rather popular way of writing the Dirac operator in physics, from a purely mathematical point of view it is mixing Clifford algebra representations with matrix representations and, furthermore, complicates the notation for higher dimension. Therefore, we will stick to the purely Clifford algebra notation in this paper, so that our results apply not only to $\R^3$ but also to $\R^n$ for any $n$. 

Furthermore, we have the Cauchy kernel 
\begin{equation}\label{EQ:k}        
k(x)=-\frac{1}{\omega_n}\frac{\overline{x}}{|x|^n},
\end{equation}
where $\omega_n=\frac{\Gamma\left(\frac{n}{2} \right)}{\pi^{n/2}}$, denotes the surface area of the $n$-dimensional unit ball, $n\geq 2$, see e.g. \cite{BDS82}. This kernel satisfies
$$
Dk=\delta,
$$
with $\delta$ being the delta-distribution. This gives rise to the so-called Teodorescu transform
\begin{equation}\label{EQ:Te}  
Tf(x)=\int_{\mathbb{R}^n} k(x-y)f(y)dy
\end{equation}
with $DT=I$. 

\subsection{Clifford-Hilbert modules}

A right (unitary) module over $\BR_{0,n}$ is a vector space $V$ together with an algebra morphism $R: \BR_{0,n} \to \mathrm{End}(V)$, or to say it more explicitly, there exists a linear transformation (also called right multiplication) $R(a)$ of $V$ such that
\begin{equation} \label{Eq:02}
R(ab+c)=R(b)R(a)+R(c),
\end{equation}
for all $a\in\BR_{0,n}$, and where $R(1)$ is the identity operator. {We consider in $V$ the right multiplication defined by 
\begin{equation} \label{Eq:03}
R(a)v =v a, \quad v \in V, a \in \BR_{0,n}.
\end{equation} In particular, if $V$ denotes a function space, the product (\ref{Eq:03}) is defined by point-wise multiplication.}
We say that $V$ is a right Banach $\BR_{0,n}$-module if
\begin{itemize}
\item $V$ is a right $\BR_{0,n}$-module;
\item $V$ is a real Banach space;
\item there exists $C>0$ such that for any $a\in\BR_{0,n}$ and $x\in V$ it holds that
\begin{equation}\label{CBM}
\|x a\|_{V}\leq C~|a|\|x\|_{V},\quad  \text{{where }} |a|^2 := \sum_A |a_A|^2.
\end{equation}
\end{itemize} In particular, we have $\|x a\|_{V}=|a|\|x\|_{V}$ if $a\in\BR$. These considerations give rise to the adequate right modules of $\BR_{0,n}$-valued functions defined over any suitable subset $\Omega$ of $\BR^{n}.$ Of course, by similar reasoning one can define adequate left modules of $\BR_{0,n}$-valued functions.\smallskip

Consider $\cH$ to be a real Hilbert space. Then $V:=\cH \otimes \BR_{0,n}$ defines a right-Clifford-Hilbert module (rHm for short). Indeed, the inner product $\langle\cdot,\cdot\rangle_{\cH}$ in $\cH$ gives rise to two inner products in $V$:
\begin{equation}\label{Eq:05}
\langle x,y\rangle_{\mathcal{H}} := \sum_{A,B} \langle x_A,y_B\rangle_{\cH} \overline{e}_A e_B \quad{\rm and}\quad \inner{x,y}  :=  \sum_A  \langle x_A,y_B\rangle_{\cH} = \rpart{ \langle x,y \rangle_{\mathcal{H}} }.
\end{equation}
We remark that while only the second inner product gives rise to a norm (in the classic sense),  the first provides a generalization of Riesz' representation theorem in the sense that a linear functional $\phi$ is continuous if and only if it can be represented by an element $f_\phi\in V$ such that $\phi(g)=\langle f_\phi,g\rangle_{\mathcal{H}}.$ Furthermore, a mapping {$K : V \rightarrow W$} between two right-Clifford-Hilbert modules $V$ and $W$ is called a $\BR_{0,n}$-linear mapping if
$K(fa+g)=K(f)a+K(g),$ {where $f, g \in V, a \in \BR_{0,n}.$} For more details we refer to~\cite{BDS82}. \\

We give some important inequalities involving the sesquilinear form and the norm coming from the complex-valued inner product:
\begin{itemize}
\item  $| \inner{f,g} | \leq K_n\|f\|_{L^p(\Omega,\BR_{0,n})} \|g\|_{L^q(\Omega,\BR_{0,n})}$ with $\frac{1}{p}+\frac{1}{q}=1$ (H\"older inequality);
\item $\|a f\|_{L^2(\Omega,\BR_{0,n})}\leq K_n |a|\|f\|_{L^2(\Omega,\BR_{0,n})}$ for all $a\in\BR_{0,n}$, while $\|a f\|_{L^2(\Omega,\BR_{0,n})}=|a|\|f\|_{L^2(\Omega,\BR_{0,n})}$ whenever $a$ is a scalar or a vector;
\item $\|f\|_{L^2(\Omega,\BR_{0,n})}\leq | \inner{f,f}|^{1/2}\leq K_n \|f\|_{L^2(\Omega,\BR_{0,n})}$;
\item $\|f\|_{L^2(\Omega,\BR_{0,n})}\leq \sup_{\|g\|_{L^2(\Omega,\BR_{0,n})} \leq 1}|\inner{f,g}|\leq K_n\|f\|_{L^2(\Omega,\BR_{0,n})}$,
\end{itemize}
where the constant $K_n$ is given by \eqref{Kn}.  
Many facts from classic Hilbert spaces carry over to the notion of a Clifford-Hilbert module. Let us also point out that exactly the same formulae hold for functions in $L^p(\Omega, \mathbb{C}_n)$, with the value of $K_n$ coming from \eqref{Knc}.

In the same way we can introduce $\mathcal{S}$ as the corresponding Schwartz space of rapidly decaying functions. Its dual space $\mathcal{S}^\prime$  given by the continuous linear functionals is the space of tempered distributions.  Again, this can be either defined component-wise or via the sesquilinear form 
$\langle f,g \rangle_{L^{2}}$, but the space $\mathcal{S}^\prime$ is again considered as a right module. Let us remark that, strictly speaking, if we consider $\mathcal{S}$ as a Fr\'echet right module, then its algebraic dual $\mathcal{S}^\prime$ is the space of all left-$\mathbb{R}_{0,n}$-linear functionals over $\mathcal{S}$, which can also be identified with elements of a right-linear module by means of the standard anti-automorphism in the above mentioned way.
 
For more details we refer to the classic books~\cite{BDS82,DSS92}.

We also need the so-called Dirac-Sobolev spaces $W_p^k(\Omega,\BR_{0,n})$ that were discussed for $n=3$ in~\cite{IS10}. These spaces are defined to be the completion of $C^\infty(\Omega,\BR_{0,n})$ under the norm
$$
\|f\|_{W_p^k(\Omega, \R_{0,n})}=\left(\|f\|_{L^p(\Omega,\BR_{0,n})}^p+\sum_{m=1}^k\|D^mf\|_{L^p(\Omega,\BR_{0,n})}^p\right)^{1/p}.
$$
For the case of $k=1$ in~\cite{IS10} it was proved that for $1<p<\infty$ these norms are equivalent to the norms  
$$
\|f\|_{W_p^k(\Omega, \R_{0,n})}=\left(\|f\|_{L^p(\Omega,\BR_{0,n})}^p+\sum_{m=1}^k\|\nabla^m f\|_{L^p(\Omega,\BR_{0,n})}^p\right)^{1/p}\,,
$$
where $\nabla$ denotes the gradient. For the case $p=1$ it was shown that the norms are not equivalent. This can be easily extended for the case of $k>1$. Furthermore, combining this results with the result from~\cite{Stein70} (Theorem 3 on page 135) we have that the norm of the Dirac-Sobolev space $W_p^k(\Omega,\BR_{0,n})$ above is equivalent to the norm of the corresponding Bessel potential space for any positive integer $k$ in the case $1<p<\infty$, but not for $p=1$ or $p=\infty$. 


\section{Sobolev inequalities}\label{SEC:sob.ineq}
From now on we will consider spaces of the type $L^p(\BR^n,\mathbb{R}_{0,m})$ and $L^p(\BR^n,\mathbb{C}_{m})$ with $m\geq n$. It is useful to allow for different values of $n$ and $m$ in view of our applications to zero modes of the Dirac operator in the last section (where we have $m=n+1$).

For the Hardy-Littlewood-Sobolev inequality we have
\begin{thm}\label{THM:Sobolev1}
Let $1<p,q<\infty$, $1/p+1/q+\lambda/n=2$ and $0<\lambda<n$. Let $m\geq n\geq 1$. Then for any $f\in L^p(\BR^n,\mathbb{R}_{0,m}) ,g\in L^q(\BR^n, \R_{0,m})$ there exists a constant $C_{p,\lambda,n,m}>0$ such that
$$
\left| \int_{\BR^n}\int_{\BR^n} \overline{f(x)}\frac{1}{|x-y|^\lambda}g(y)dxdy\right|\leq C_{p,\lambda,n,m}\|f\|_{L^p(\BR^n, \R_{0,m})}\|g\|_{L^q(\BR^n, \R_{0,m})}.
$$
The same result holds for $f\in L^p(\BR^n,\mathbb{C}_{m})$ and $g\in L^q(\BR^n, \mathbb{C}_{m})$.
\end{thm}

\begin{proof}
Under the conditions of the theorem we have 
$$
\left| \int_{\BR^n}\int_{\BR^n} \overline{f(x)}\frac{1}{|x-y|^\lambda}g(y)dxdy\right|\leq K_m \int_{\BR^n}\int_{\BR^n} |\overline{f(x)}| \frac{1}{|x-y|^\lambda}|g(y)|dxdy\,,
$$
where we used the inequality $|ab|\leq K_m|a||b|$ ($a,b\in\mathbb{R}_{0,m}$) for Clifford-valued functions:
$$
\left|\overline{f(x)}\frac{1}{|x-y|^\lambda}g(y)\right|\leq K_m\left|\overline{f(x)}\right|\left|\frac{1}{|x-y|^\lambda}g(y)\right|\leq K_m\left|\overline{f(x)}\right|\frac{1}{|x-y|^\lambda}\left|g(y)\right|.
$$
Now, we apply the classic Hardy-Littlewood-Sobolev inequality~\cite{HL28,Sob38}
$$
\int_{\BR^n}\int_{\BR^n} |\overline{f(x)}| \frac{1}{|x-y|^\lambda}|g(y)|dxdy\leq C_{p,\lambda,n,m}\|f\|_{L^p}\|g\|_{L^q}
$$
to obtain our theorem. 
\end{proof}

The above theorem allows us to get a Hardy-Littlewood-Sobolev inequality for the Cauchy kernel instead of the fractional Riesz kernel by simply using $\left|\frac{\overline{x-y}}{|x-y|^n}\right|=\frac{1}{|x-y|^{n-1}}$. Furthermore, in the case of $n=1$, we have the classic H\"older inequality for Clifford-valued functions. 
\begin{cor}
Let $1<p,q<\infty$, $1/p+1/q-1/n=1$. Let $m\geq n\geq 1$. Then for any $f\in L^p(\BR^n,\R_{0,m}) ,g\in L^q(\BR^n,\R_{0,m})$ there exists a constant $C_{p,n-1,n,m}>0$ such that
$$
\left| \int_{\BR^n}\int_{\BR^n} \overline{f(x)}\frac{\overline{x-y}}{|x-y|^n}g(y)dxdy\right|\leq C_{p,n-1,n,m}\|f\|_{L^p(\R^n, \R_{0,m})}\|g\|_{L^q(\R^n,\R_{0,m})}\,,
$$
where $C_{p,n-1,n,m}>0$ is as in Theorem \ref{THM:Sobolev1} for $\lambda=n-1$. The same result holds also for $f\in L^p(\BR^n,\mathbb{C}_{m})$ and $g\in L^q(\BR^n, \mathbb{C}_{m})$.
\end{cor}

Let us remark that we can also get concrete values for the constant $C_{p,\lambda,n}$ in some special cases using the classic results from~\cite{Li83} (Theorem 3.1 and Corollary 3.2). These constants will be also useful for our further analysis.

\begin{thm}\label{THM:HLS-e} Let $m\geq n\geq 1$.
Let $p=q=\frac{2n}{2n-\lambda}$ and $0<\lambda<n$. Then for any $f,g\in L^p(\BR^n,\R_{0,m})$ we have
$$
\left| \int_{\BR^n}\int_{\BR^n} \overline{f(x)}\frac{1}{|x-y|^\lambda}g(y)dxdy\right|\leq C_{p,\lambda,n,m}\|f\|_{L^p(\R^n,\R_{0,m})}\|g\|_{L^p(\R^n,\R_{0,m})}
$$
with $C_{p,\lambda,n,m}=K_m\pi^{\lambda/2}\frac{\Gamma(n/2-\lambda/2)}{\Gamma(n-\lambda/2)}\left(\frac{\Gamma(n/2)}{\Gamma(n)}\right)^{-1+\lambda/n}$ for $K_m$ as in \eqref{Kn}. 

If $q=2$ and $p=2n/(3n-2\lambda)$ with $n<2\lambda<2n$, then for any $f\in L^p(\BR^n,\R_{0,m})$, $g\in L^2(\BR^n,\R_{0,m})$, we have
$$
\left| \int_{\BR^n}\int_{\BR^n} \overline{f(x)}\frac{1}{|x-y|^\lambda}g(y)dxdy\right|\leq C_{p,\lambda,n,m}\|f\|_{L^p(\R^n,\R_{0,m})}\|g\|_{L^2(\R^n,\R_{0,m})}
$$
with $C_{p,\lambda,n,m}=K_m\pi^{\lambda/2}\frac{\Gamma(n/2-\lambda/2)}{\Gamma(n-\lambda/2)}\left(\frac{\Gamma(\lambda-n/2)}{\Gamma(3n/2-\lambda)}\right)^{1/2} \left(\frac{\Gamma(n/2)}{\Gamma(n)}\right)^{-1+\lambda/n}$ . 
The same result holds for $f,g\in L^p(\BR^n,\mathbb{C}_{m})$, with $K_m$ as in \eqref{Knc}.
\end{thm}

For the specific case of the Cauchy kernel we have the following corollary.

\begin{cor} Let $m\geq n\geq 2$.
Let $p=q=\frac{2n}{n+1}$. Then for any $f,g\in L^p(\BR^n, \R_{0,m})$ we have
$$
\left| \int_{\BR^n}\int_{\BR^n} \overline{f(x)}\frac{\overline{x-y}}{|x-y|^{n}}g(y)dxdy\right|\leq C_{p,n-1,n,m}\|f\|_{L^p(\BR^n, \R_{0,m})}\|g\|_{L^p(\BR^n, \R_{0,m})}
$$
with $C_{p,n-1,n,m}=K_m\pi^{(n-1)/2}\frac{\Gamma(1/2)}{\Gamma((n+1)/2)}\left(\frac{\Gamma(n)}{\Gamma(n/2)}\right)^{1/n}$.

In the case $n\geq 3$ if $q=2$ and $p=2n/(n+2)$ then for any $f\in L^p(\BR^n, \R_{0,m})$, $f \in L^2(\BR^n, \R_{0,m})$ we have
$$
\left| \int_{\BR^n}\int_{\BR^n} \overline{f(x)}\frac{\overline{x-y}}{|x-y|^{n}}g(y)dxdy\right|\leq C_{p,n-1,n,m}\|f\|_{L^p(\BR^n, \R_{0,m})}\|g\|_{L^2(\BR^n, \R_{0,m})}
$$
with $C_{p,n-1,n,m}=K_m\pi^{(n-1)/2}\frac{\Gamma(1/2)}{\Gamma((n+1)/2)}\left(\frac{\Gamma(n/2-1)}{\Gamma(n/2+1)}\right)^{1/2} \left(\frac{\Gamma(n)}{\Gamma(n/2)}\right)^{1/n}$ . In both cases the constant $K_m$ is as in \eqref{Kn}.

The same result holds for $f,g\in L^p(\BR^n,\mathbb{C}_{m})$, with $K_m$ as in \eqref{Knc}.
\end{cor}

Note that since we have $D^2=-\Delta$ we also have $$\|Df\|^{2}_{L^2}=\inner{Df,Df}=\inner{-\Delta f,f}=\|\nabla f\|^{2}_{L^2},$$
since $D^*=-D$.

As a consequence, we obtain the following Sobolev inequality for the Dirac operator, with an explicit constant.

\begin{thm}\label{THM:Sobolev2} Let $m\geq n\geq 3$. Then we have
$$
\|f\|_{L^{\frac{2n}{n-2}}(\BR^n, \R_{0,m})}\leq C_1 \|D f\|_{L^2(\BR^n, \R_{0,m})},
$$
with 
$C_1=\frac{\sqrt{2}K_m^{3/2}}{\sqrt{n(n-2)}}\pi^{\frac{n-1}{2}}\frac{\Gamma(n)^{1/n}}{\Gamma(n/2)^{1+1/n}}$ where $K_m$ is  as in \eqref{Kn}.

The same result holds for $f\in L^p(\BR^n,\mathbb{C}_{m})$ with $K_m$ as in \eqref{Knc}.
\end{thm}

\begin{proof}
First we observe that using the Plancherel formula for the Fourier transform of Clifford-valued functions~(\cite{BDS82}) we get  

\begin{eqnarray}\label{inneru,g}
|\inner{u,g}|^2 & = & \left|\inner{\widehat u, \widehat g}\right|^2\nonumber\\
& = & \left|\inner{(i\xi)\widehat u, \frac{i\xi}{|\xi|^2}\widehat g}\right|^2\nonumber\\
& \leq & K_m^2 \|D u\|^2_{L^2}\|Tg\|^2_{L^2}\nonumber\\
& = & K_m^2\inner{Du,Du}\inner{Tg,Tg}.
\end{eqnarray}
Let us now estimate the term $\inner{Tg,Tg}$, with $T$ as in \eqref{EQ:k}-\eqref{EQ:Te}. We have
\begin{equation*}
    \begin{split}
      \inner{Tg,Tg}  &= \left[\int_{\R^n}\int_{\R^n}\overline{\frac{1}{\omega_n}\frac{x-z}{|x-z|^n}g(x)dx}\int_{\R^n}\frac{1}{\omega_n}\frac{y-z}{|y-z|^n}g(y)dydz \right]_0\\&
\left[\int_{\R^n}\int_{\R^n}\overline{g(x)}\left(\int_{\R^n}\frac{1}{\omega_{n}^{2}}\frac{\overline{x-z}}{|x-z|^n}\frac{y-z}{|y-z|^n}dz \right)g(y)dxdy \right]_0.
    \end{split}
\end{equation*}
With the change of variables $v=z-y$ the above integral with respect to the variable $z$ becomes 
\begin{equation*}
    -\int_{\R^n}\frac{1}{\omega_{n}^{2}}\frac{\overline{x-y-v}}{|x-y-v|^n}\frac{v}{|v|^n}dv\,,
\end{equation*}
and we also have 
\begin{equation*}
    \frac{v}{|v|^n}=D\left(\frac{1}{(-n+2)}\frac{1}{|v|^{n-2}}\right),
\end{equation*}
in view of
$$
\sum_j e_j\partial_j\left( (\sum_{k}v_k^2)^{-\frac{n-2}{2}}\right)=
\sum_j e_j (\sum_{k}v_k^2)^{-\frac{n-2}{2}-1}(-\frac{n-2}{2})2v_j=
(-n+2)|v|^{-n}v.
$$
Therefore, we get 
\begin{equation*}
    \begin{split}
     \int_{\R^n}\frac{1}{\omega_{n}^{2}}\frac{\overline{x-y-v}}{|x-y-v|^n}\frac{v}{|v|^n}dv\ &=    \frac{1}{\omega_{n}^{2}} \int_{\R^n} \frac{\overline{x-y-v}}{|x-y-v|^n}D\left(\frac{1}{(-n+2)}\frac{1}{|v|^{n-2}}\right)dv\\&
     =\frac{1}{\omega_n}TD\left(\frac{1}{(-n+2)}\frac{1}{|\cdot|^{n-2}}\right)(x-y)\\
     & =\frac{1}{\omega_n}\frac{1}{(-n+2)}\frac{1}{|x-y|^{n-2}}.
    \end{split}
\end{equation*}
From our Hardy-Littlewood-Sobolev inequality~(Theorem~\ref{THM:HLS-e}, first part) for $\lambda=n-2$, we have
$C_{p,n-2,n,m}=K_m\pi^{\frac{n-2}{2}}\frac{1}{\Gamma(\frac{n+2}{2})}\left(\frac{\Gamma(n/2)}{\Gamma(n)}\right)^{-\frac{2}{n}}$. 
Consequently, for 
the case of $p=\frac{2n}{n+2}$ and for $\lambda=n-2$ with $\omega_n=\frac{\Gamma(n/2)}{\pi^{n/2}}$ we get
\begin{eqnarray*}
\inner{Tg, Tg} & = & \frac{1}{(n-2)\omega_n}\int_{\R^n}\int_{\R^n} \left[\overline{g(x)}\frac{1}{|x-y|^{n-2}} g(y)\right]_0dxdy\\
& \leq & \frac{1}{(n-2)\omega_n}C_{p,n-2,n,m} \|g\|^2_{L^\frac{2n}{n+2}}\\
 & = & \frac{2 K_m}{n(n-2)}\pi^{n-1}\frac{\Gamma(n)^{2/n}}{\Gamma(n/2)^{2+2/n}}\|g\|^2_{L^\frac{2n}{n+2}}\,,
\end{eqnarray*}
where we have used the fact that $\Gamma(x+1)=x\Gamma(x)$.
Now, using $|u|^2=[u\overline{u}]_0$ and \eqref{inneru,g} for $q>1$, we have $|u|^q=|u|^{q-2}|u|^2=|u|^{q-2}[u\bar{u}]_0$, and hence
\begin{eqnarray*}
    \|u\|_{L^q}^{2q} & = & (\int_{\R^n} [u|u|^{q-2}\bar u]_0 dx)^2\\
    & = & |\inner{u, |u|^{q-2} \overline{u}}|^2\\
    & \leq & K_m^2 \inner{Du,Du} \inner{T(|u|^{q-2} \overline{u}), T(|u|^{q-2} \overline{u})}\\
    & \leq & \frac{2 K_m^3}{n(n-2)}\pi^{n-1}\frac{\Gamma(n)^{2/n}}{\Gamma(n/2)^{2+2/n}}   \inner{Du,Du} \||u|^{q-2} \overline{u}\|_{L^\frac{2n}{n+2}}^{2}.
\end{eqnarray*}
Obviously, in the case of $q>1$ we have that $||u|^{q-2} \overline{u}|\leq |u|^{q-1}$ is well defined.
Now, we claim that for $q=\frac{2n}{n-2}$, we have that 
$|u|^{q-2} \overline{u}\in L^\frac{2n}{n+2}(\mathbb{R}^n,\mathbb{R}_{0,m})$ with the equality of the norm. Indeed, we have $q-1=\frac{n+2}{n-2}$ and
\begin{multline*}
    \left(\int_{\R^n}\left| |u|^{q-2} \overline{u}\right|^{\frac{2n}{n+2}}dx\right)^{\frac{n+2}{2n}} =
    \left(\int_{\R^n}|u|^{\frac{2n}{n-2}}dx\right)^{\frac{n+2}{2n}}
    \\ = \left(\int_{\R^n}|u|^{\frac{2n}{n-2}}dx\right)^{\frac{n-2}{2n}\frac{2n}{n-2}\frac{n+2}{2n}}=
    \|u\|_{L^\frac{2n}{n-2}}^{\frac{n+2}{n-2}}.
\end{multline*}
Consequently, with $q=\frac{2n}{n-2}$, we have
\begin{eqnarray*}
\|u\|_{L^q}^{2q} & \leq & \frac{2 K_m^3}{n(n-2)}\pi^{n-1}\frac{\Gamma(n)^{2/n}}{\Gamma(n/2)^{2+2/n}}   \inner{Du,Du} \||u|^{q-2} \overline{u}\|_{L^\frac{2n}{n+2}}^{2} \\
& = &
\frac{2K_m^3}{n(n-2)}\pi^{n-1}\frac{\Gamma(n)^{2/n}}{\Gamma(n/2)^{2+2/n}}   \inner{Du,Du} \|u\|_{L^q}^{2\frac{n+2}{n-2}}.
\end{eqnarray*}
Since for $q=\frac{2n}{n-2}$ we have that $2q-2\frac{n+2}{n-2}=2$, we obtain
\begin{eqnarray*}
 \|u\|_{L^\frac{2n}{n-2}}^{2} & \leq & \frac{2K_m^3}{n(n-2)}\pi^{n-1}\frac{\Gamma(n)^{2/n}}{\Gamma(n/2)^{2+2/n}}  \inner{Du,Du}\\
& = & \frac{2K_m^3}{n(n-2)}\pi^{n-1}\frac{\Gamma(n)^{2/n}}{\Gamma(n/2)^{2+2/n}} \|Du\|^{2}_{L^2}\,,
\end{eqnarray*}
completing the proof.
\end{proof}
Next we will prove the $L^p(\R^n, \R_{0,m})$-Sobolev inequality, for $1< p<n$, using a different method. Before doing so, let us first recall the following version of Young's convolution inequality, see e.g. \cite[Proposition 1.5.2]{FR16}. If $p,q \in (1,\infty)$ are such that $\frac{1}{p}+\frac{1}{q}>1$, $f_1 \in L^p(\R^n)$ and $f_2$ satisfies the weak-$L^q(\R^n)$ condition, i.e., it satisfies
\[
\sup_{s>0} s^q |\{ x\,:\, |f_2(x)|>s\}| := \|f_2\|^{q}_{w-L^q(\R^n)}<\infty\,,
\]
then we have 
\begin{equation}\label{Young2}
    \|f_1 \ast f_2\|_{r} \leq \|f_1\|_p \|f_2\|_{w-L^q(\R^n)}<\infty\,,
\end{equation}
for $r$ such that $\frac{1}{p}+\frac{1}{q}=\frac{1}{r}+1$. It can easily be verified that the convolution inequality \eqref{Young2} can be extended to Clifford-valued functions with the inequality being replaced by
\begin{equation}
    \label{Young}
    \|f_1 \ast f_2\|_{r} \leq K_m \|f_1\|_p \|f_2\|_{w-L^q(\R^n)}<\infty\,,
\end{equation}
in the general case, and by
\begin{equation}
    \label{Young1}
    \|f_1 \ast f_2\|_{r} \leq \|f_1\|_p \|f_2\|_{w-L^q(\R^n)}<\infty\,,
\end{equation}
in case either $f_1$ or $f_2$ is scalar or vector-valued. 

Using this inequality we can prove the following theorem.
\begin{thm}
    \label{THM:Sob.p}
    Let $l\geq 1$ be an integer, and let $1< p<n/l$, and let $m\geq n\geq 2$.
    Then we have
    \[
    \|f\|_{L^{\frac{pn}{n-pl}}(\BR^n,\mathbb{R}_{0,m})}\leq C_l \|D^lf\|_{L^p(\BR^n,\mathbb{R}_{0,m})}\,,
    \]
    with $C_l=c_l n^{\frac{l-n}{n}}\omega_n^{-\frac{l}{n}},$
    where $c_l=\frac{1}{2^{j-1}(j-1)!\Pi_{h=1}^j(2h-n)}$ if $l=2j$ is even and $c_l=\frac{1}{2^{j} j!\Pi_{h=1}^j(2h-n)}$ if $l=2j+1$ is odd. The inequality continues to be true for $f\in L^p(\BR^n,\mathbb{C}_{m})$.
\end{thm}
In particular, for $l=1$, we have $c_1=1$, and hence
\begin{equation}\label{EQ:c1-2}
    C_1= n^{\frac{1-n}{n}}\omega_n^{-\frac{1}{n}}
    =n^{\frac{1-n}{n}} \Gamma\left(\frac{n}{2} \right)^{-\frac{1}{n}}\pi^{1/2},
\end{equation}
in view of
$\omega_n=\frac{\Gamma\left(\frac{n}{2} \right)}{\pi^{n/2}}$.

\begin{rem}\label{REM:C1}
Investigating values of $C_1$ in Theorem \ref{THM:Sobolev2} and in Theorem  \ref{THM:Sob.p}, we can find
that the constant in \eqref{EQ:c1-2} is smaller for small $n$, while for large $n$ the constant $C_1$ in Theorem \ref{THM:Sobolev2} becomes smaller. Therefore, in what follows, when we refer to $C_1$, we will take the smaller of these two constants, from Theorem \ref{THM:Sobolev2} and from \eqref{EQ:c1-2}, which will depend on the value of $n$, so that with this (smaller) $C_1$ we have the inequality
$$
\|f\|_{L^{\frac{2n}{n-2}}(\BR^n, \R_{0,m})}\leq C_1 \|D f\|_{L^2(\BR^n, \R_{0,m})}.
$$
\end{rem}

\begin{proof}[Proof of Theorem \ref{THM:Sob.p}]
    Recall that the fundamental solution $k_l$ for the Dirac operator $D^{l}$, for $l\geq 1$ integer, is of the form~(\cite{KFKC12})
    \[
    k_l(x)=-\frac{c_l}{\omega_n}\frac{\overline{x}^l}{|x|^{n}},
    \]
    with $c_l=\frac{1}{2^{j-1}(j-1)!\Pi_{h=1}^j(2h-n)}$ if $l=2j$ is even and $c_l=\frac{1}{2^{j}j!\Pi_{h=1}^j(2h-n)}$ if $l=2j+1$ is odd. This can also be easily deduced from the formulae
    \[
    D\frac{\overline{x}^{2j}}{|x|^n}=(2j-n)\frac{\overline{x}^{2j-1}}{|x|^n} \mbox{ and } D\frac{\overline{x}^{2j+1}}{|x|^n}=(2j)\frac{\overline{x}^{2j}}{|x|^n},
    \]
    with $j\in\mathbb{N}$.
    Then the solution $f$ to the equation $D^lf=u$ can be written as the convolution $f=u \ast k_l$.
    We see that $k_l \in w-L^{\frac{n}{n-l}}$, since 
    \begin{multline*}
     |\{ x : k_l(x)>s\}|= |\{ x: \frac{c_l}{\omega_n}|x|^{l-n}>s\}|= \\
     |\{ x: |x|< \left(\frac{\omega_n}{c_l}\right)^{-\frac{1}{n-l}} s^{-\frac{1}{n-l}}\}|=\frac{\omega_n}{n}
     \left(\frac{\omega_n}{c_l}\right)^{-\frac{n}{n-l}}s^{-\frac{n}{n-l}},
     \end{multline*}
     in view of $|\{x: |x|<R\}|=\int_{B_R}dx=\frac{\omega_n}{n}R^n$, and hence
 $$
 \|k_l\|_{w-L^{\frac{n}{n-l}}}=\left(\sup_{s>0} s^{\frac{n}{n-l}}
 |\{ x : k_l(x)>s\}|\right)^{\frac{n-l}{n}}=
 \left( \frac{\omega_n}{n}\right)^{\frac{n-l}{n}}\frac{c_l}{\omega_n}
 =\frac{c_l}{n^{\frac{n-l}{n}}\omega_n^{\frac{l}{n}}}.
 $$
    Now, since $k_l$ is vector-valued ($\overline x^l$ is always a vector in view of $x^2=-|x|^2$ for $x\in\mathbb{R}^n$) we can use Young's inequality in the form \eqref{Young1} and we obtain 
    \begin{multline*}
        \|f\|_{L^r}  =   \|u \ast k\|_{L^r}
         \leq \|u\|_{L^p} \|k\|_{w-L^{\frac{n}{n-l}}}
         = \|D^lf\|_{L^p} \|k\|_{w-L^{\frac{n}{n-l}}}=
        \frac{c_l }{n^{\frac{n-l}{n}}\omega_n^{\frac{l}{n}}}\|D^lf\|_{L^p}\,,
    \end{multline*}
    where $r=\frac{pn}{n-pl}$, and the result follows.
\end{proof}

\section{Logarithmic Sobolev and Poincar\'e inequalities}\label{SEC:log.SobandPoi}
The following result, see \cite[Lemma 3.2]{CKR24} is the logarithmic version of the classical H\"older inequality, cf. \cite[Theorem 5.5.1 (ii)]{G92} on general measure spaces. 

\begin{lem}[Logarithmic H\"older inequality]
    Let $\mathbb{X}$ be a  measure space and let $u\in L^{p}(\mathbb{X})\cap L^{q}(\mathbb{X})\setminus\{0\}$, where $1<p<q< \infty.$ 
Then we have
\begin{equation}\label{holdernn}
\int_{\mathbb{X}}\frac{|u|^{p}}{\|u\|^{p}_{L^{p}(\mathbb{X})}}\log\left(\frac{|u|^{p}}{\|u\|^{p}_{L^{p}(\mathbb{X})}}\right)dx\leq \frac{q}{q-p}\log\left(\frac{\|u\|^{p}_{L^{q}(\mathbb{X})}}{\|u\|^{p}_{L^{p}(\mathbb{X})}}\right).
\end{equation}
\end{lem}

\begin{cor}[Logarithmic Sobolev inequality]\label{COR:log.sob}
Let $m\geq n\geq 3$.
The logarithmic Sobolev inequality is as follows: 
\begin{equation}
    \label{EQ:log.sob(thm)}
  \int_{\R^n} \frac{|u|^2}{\|u\|_{L^2(\R^n,\mathbb{R}_{0,m})}^2}\log \left(\frac{|u|}{\|u\|_{L^2(\R^n,\mathbb{R}_{0,m})}} \right)\,dx \leq \frac{n}{2} \log \left(\frac{C_{1}\|Du\|_{L^2(\R^n,\mathbb{R}_{0,m})}}{\|u\|_{L^2(\R^n,\mathbb{R}_{0,m})}} \right)
\end{equation}
for all $u \in L^2(\BR^n,\mathbb{R}_{0,m}) \setminus \{0\}$, where $C_1$ is as in Remark \ref{REM:C1}.
The same result holds for $u \in L^2(\BR^n,\mathbb{C}_{m}) \setminus \{0\}$ with $K_m$ being replaced by the corresponding constant in the case of $\mathbb{C}_{m}$  as in \eqref{Knc}.
\end{cor}
\begin{proof}
   From \eqref{holdernn} for $p=2$, $q \in (p, \infty)$ and $\mathbb{X}=\R^n$ we have
\[
\int_{\mathbb{R}^n} \frac{|u|^{2}}{\|u\|^{2}_{L^{2}}}\log\left(\frac{|u|^{2}}{\|u\|^{2}_{L^{2}}}\right)dx\leq \frac{q}{q-2}\log\left(\frac{\|u\|^{2}_{L^{q}}}{\|u\|^{2}_{L^{2}}}\right)\,.
\]
Choosing $q=\frac{2n}{n-2}$ by Theorem \ref{THM:Sobolev2} we can further estimate
\begin{eqnarray*}
    \int_{\mathbb{R}^n} \frac{|u|^{2}}{\|u\|^{2}_{L^{2}}}\log\left(\frac{|u|^{2}}{\|u\|^{2}_{L^{2}}}\right)dx & \leq & \frac{n}{2}\log \left(\frac{C_{1}^{2} \|Du\|^{2}_{L^2}}{\|u\|_{L^2}^2} \right)\,,
\end{eqnarray*}
and the proof is complete. 
\end{proof}
\begin{thm}\label{THM:Nash}
Let $m\geq n\geq 3$.
For  $u \in C_{0}^{\infty}(\R^n, \R_{0,m})$ or $u \in C_{0}^{\infty}(\R^n, \mathbb{C}_{m})$  the Nash inequality with respect to the Dirac operator $D$ is as follows
\begin{equation}
    \label{EQ:Nash(thm)}
    \|u\|_{L^2(\R^n,\mathbb{R}_{0,m})}^{1+2/n}\leq  C_{1}\|u\|_{L^1(\R^n,\mathbb{R}_{0,n})}^{2/n}\|Du\|_{L^2(\R^n,\mathbb{R}_{0,m})},
\end{equation}
where $C_1$ is as in Remark \ref{REM:C1}.
\end{thm}
\begin{proof}
    Let us define the auxiliary function $f(r)=\log \left(\frac{1}{r} \right)$, for $r>0$. We have
    \begin{eqnarray}\label{EQ:nash.1}
        \log \left(\frac{\|u\|_{L^2}^2}{\|u\|_{L^1}} \right) & = & f \left( \frac{\|u\|_{L^1}}{\|u\|_{L^2}^2}\right)\nonumber\\
        & = & f \left( \int_{\R^n} \frac{1}{|u|}|u|^2 \frac{1}{\|u\|^{2}_{L^2}}\,dx\right)\nonumber\\
        & \leq & \int_{\R^n} f\left(\frac{1}{|u|} \right)\frac{|u|^2}{\|u\|_{L^2}^2}\,dx \nonumber\\
        & = &  \int_{\R^n} \log (|u|)\frac{|u|^2}{\|u\|_{L^2}^2}\,dx\,,
    \end{eqnarray}
    where we have applied Jensen's inequality for the function $f$ and for the probability measure $\frac{|u|^2}{\|u\|_{L^2}^2}\,dx$.
    
    Now,  the logarithmic Sobolev inequality \eqref{EQ:log.sob(thm)} implies
    \begin{equation}\label{EQ:nash.2}
        \int_{\R^n} \frac{|u|^2}{\|u\|_{L^2}^2}\log (|u|)\,dx \leq \log (\|u\|_{L^2})+\frac{n}{2} \log \left(C_{1}\frac{\|Du\|_{L^{2}}}{\|u\|_{L^2}}\right)\,.
    \end{equation}
    A combination of \eqref{EQ:nash.1} and \eqref{EQ:nash.2}, after rearrangement, gives
    \[
    \log \left( \|u\|_{L^2}\right)+\frac{n}{2}\log (\|u\|_{L^2}) \leq \frac{n}{2} \log (C_{1}\|u\|_{L^1}^{2/n}\|Du\|_{L^2})\,,
    \]
    which in turn implies
    \[
    \left(1+\frac{n}{2} \right) \log \left( \|u\|_{L^2}\right) \leq \frac{n}{2} \log (C_{1}\|u\|_{L^1}^{2/n}\|Du\|_{L^2})\,,
    \]
    and so, by the monotonicity of the logarithm, we conclude
    \[
    \|u\|_{L^2}^{1+2/n}\leq  C_{1}\|u\|_{L^1}^{2/n}\|Du\|_{L^2} \,,
    \]
    and the result follows.
\end{proof}
\begin{cor}
Let $m\geq n\geq 3$.
    Let $f_0 \in L^1(\R^n,\mathbb{R}_{0,m}) \cap L^2(\R^n,\mathbb{R}_{0,m})$ be such that $f_{0,j}\geq 0$ for all $j$. Then the solution to the heat equation 
    \[
    \partial_t f=D^2 f\,,\quad f(0,x)=f_0(x)\,,
    \]
    satisfies the estimate
    \[
      \|f(t,\cdot)\|_{L^2(\R^n,\mathbb{R}_{0,m})} \leq \left( \frac{16}{nC_1^2}  \|f_0\|_{L^1(\R^n,\mathbb{R}_{0,m})}^{-4/n} t+ \|f_0\|_{L^2(\R^n,\mathbb{R}_{0,m})}^{-4/n}\right)^{-n/4},
    \]
    where $C_1$ is as in Remark \ref{REM:C1}. The same estimate continues to be true in the case of  $f_0 \in L^1(\R^n,\mathbb{C}_{m}) \cap L^2(\R^n,\mathbb{C}_{m})$.
\end{cor}

\begin{proof}
    Since we assumed that $f_{0,j}(x) \geq 0$, by the positivity of the heat kernel $h_t$ of the Laplacian, we immediately see that $f_{j}(t,x)\geq 0$, for all $j$. Moreover, the corresponding integrals are equal since
    \[
    \int_{\R^n} f_j(t,x)\,dx=\int_{\R^n} \int_{\R^n} h_t(x-y)f_{0,j}(y)\,dy\,dx=\int_{\mathbb{R}^{n}} f_{0,j}(y)dy\,, 
    \]
    where we have applied Fubini's theorem and the fact that $\|h_t\|_{L^1}=1$. For the comparison of the $L^1$-norms of the Clifford-valued functions $f$ and $f_0$, recall the following inequalities for positive numbers
    \[
    \sum_{m=1}^{N}a_{m}\leq \sqrt{N}\left(\sum_{m=1}^{N}a_{m}^{2} \right)^{1/2}\,,
    \]
    \[
    \left(\sum_{m=1}^{N} a_{m}^{2} \right)^{1/2}\leq \sum_{m=1}^{N} a_m\,.
    \]
    Using these inequalities, we have 
    \begin{eqnarray*}
        \int_{\R^n} |f(t,x)|\,dx & = & \int_{\R^n} (\sum_{j=1}^{2^n}f_j(t,x)^{2})^{1/2}\,dx\\
        & \geq & 2^{-n/2}\int_{\R^n} \sum_{j=1}^{2^n} f_{j}(t,x)\,dx\\
        & = & 2^{-n/2}\int_{\R^n} \sum_{j=1}^{2^n} f_{0,j}(x)\,dx\\
         & \geq & 2^{-n/2} \int_{\R^n} \left(\sum_{j=1}^{2^n}f_{0,j}(x)^2 \right)^{1/2}\,dx\\
          & = & 2^{-n/2}\int_{\R^n} |f_{0}(x)|\,dx\,.
    \end{eqnarray*}
    We also have 
    \begin{eqnarray*}
        \int_{\R^n} |f(t,x)|\,dx & = & \int_{\R^n} (\sum_{j=1}^{2^n}f_j(t,x)^{2})^{1/2}\,dx\\
        & \leq & \int_{\R^n} \sum_{j=1}^{2^n} f_{j}(t,x)\,dx\\
        & = & \int_{\R^n} \sum_{j=1}^{2^n} f_{0,j}(x)\,dx\\
         & \leq & 2^{n/2} \int_{\R^n} \left(\sum_{j=1}^{2^n}f_{0,j}(x)^2 \right)^{1/2}\,dx\\
          & = & 2^{n/2}\int_{\R^n} |f_{0}(x)|\,dx\,.
    \end{eqnarray*}
    Summarising the above we get 
    \begin{equation}\label{EQ:app.Nash.L1}
        2^{n/2}\|f_0\|_{L^1} \geq  \|f(t,\cdot)\|_{L^1}\geq 2^{-n/2}\|f_0\|_{L^1}\,.
    \end{equation}
    Multiplying by $f$ the heat equation and integrating over $\mathbb{R}^n$ we get 
    \[
    \langle \partial_t f(t,\cdot),f(t,\cdot)\rangle_{L^2}+\langle D^2f(t,\cdot),f(t,\cdot)\rangle_{L^2}=0\,,
    \]
    which in turn implies that 
    \[
    \frac{d}{dt}\|f(t,\cdot)\|^{2}_{L^2}=-2\|D f(t,\cdot)\|_{L^2}^2\,.
    \]
    So
    \begin{eqnarray*}
      \frac{d}{dt}\|f(t,\cdot)\|^{2}_{L^2}  & = & -2\|D f(t,\cdot)\|_{L^2}^2\\
      &\stackrel{\eqref{EQ:Nash(thm)}}\leq  & -2 C_{1}^{-2}\|f(t,\cdot)\|_{L^2}^{4/n+2}\|f(t,\cdot)\|_{L^1}^{-4/n}\\
      & \stackrel{\eqref{EQ:app.Nash.L1}}\leq  & -2^{3} C_{1}^{-2}\|f(t,\cdot)\|_{L^2}^{4/n+2}\|f_0\|_{L^1}^{-4/n}\,.
    \end{eqnarray*}
    Integrating the latter on $t \geq 0$, we get that the solution to the heat equation satisfies the estimate 
    \[
    \|f(t,\cdot)\|_{L^2} \leq  \left(\frac{16}{n} C_{1}^{-2}  \|f_0\|_{L^1}^{-4/n} t+ \|f_0\|_{L^2}^{-4/n}\right)^{-n/4}\,,
    \]
    and the proof is complete.
\end{proof}
The next result is the extension of the classical Gaussian logarithmic Sobolev inequality that was originally proved by Gross in \cite{Gro75}, and later on was also extended by the same author on more general Lie groups \cite{Gro92}. This inequality was further investigated in other sub-Riemannian situations too, cf. \cite{CKR24}, ~\cite{GL22}. 
\begin{thm}\label{THM:gauss}
Let $m\geq n\geq 3$.
If $u \in H^1(\R^n,d\mu, \R_{0,m})$ is such that $\|u\|_{L^2(\mu,\mathbb{R}_{0,m})}=1$, where $d\mu:=k e^{-\frac{|x|^2}{2}}\,dx$, with $k=\left(\frac{C_1}{2}\sqrt{ne}\right)^{n}$ for $C_1$ as in Remark \ref{REM:C1}, is the Gaussian measure, then we have the following Gaussian logarithmic Sobolev inequality with respect to the Dirac operator $D$
\begin{equation}
    \label{EQ:gauss(thm)}
    \int_{\R^n} |u|^2 \log|u|\,d\mu \leq \int_{\R^n} |Du|^2\,d\mu\,.
\end{equation}
\end{thm}
The theorem holds also in the case of $u \in H^1(\R^n,d\mu, \mathbb{C}_{m})$ with the obvious modifications. 

\begin{proof}
   It is enough to assume that $u$ is regular, see \cite[Theorem 7.3]{CKR24} for the limiting argument. For $u$ as in the hypothesis, let us define $v$ by 
    \[
    v(x)=k^{1/2}e^{-\frac{|x|^2}{4}}u(x)\,.
    \]
    Then we have $\|v\|_{L^2(\R^n, \R_{0,m})}=1$. Applying the logarithmic Sobolev inequality \eqref{EQ:log.sob(thm)} we can estimate 
    \begin{eqnarray}
        \label{EQ:gauss.1}
        \int_{\R^n} |u(x)|^2 \log|u(x)|\, d\mu & = & \int_{\R^n} |v(x)|^2 \log |k^{-1/2}e^{\frac{|x|^2}{4}}v(x)|\,dx\nonumber\\
        & = & \int_{\R^n} |v(x)|^2 \log |v(x)|\,dx\nonumber\\
        & + & \log(k^{-1/2})+\int_{\R^n} \frac{|x|^2}{4}|v(x)|^2\,dx\nonumber\\
        & \leq & \frac{n}{4}\log \left(C_{1}^2 \int_{\R^n} |Dv(x)|^2\,dx\right)\nonumber\\
        & + & \log(k^{-1/2})+\int_{\R^n} \frac{|x|^2}{4}|v(x)|^2\,dx\,.
    \end{eqnarray}
    We have
   \begin{eqnarray*}
       Du & = & k^{-1/2} (De^{\frac{|x|^2}{4}})v+k^{-1/2}e^{\frac{|x|^2}{4}} (Dv)\\
       & = & k^{-1/2} \left(\frac{x}{2}e^{\frac{|x|^2}{4}}\right)v+k^{-1/2}e^{\frac{|x|^2}{4}} (Dv)\,.
   \end{eqnarray*}
   Hence
    \begin{eqnarray}\label{EQ:gauss.2}
    \langle Du, Du \rangle_{L^2(\mu)} & = & \langle k^{-1/2} (x/2)e^{\frac{|x|^2}{4}}v, k^{-1/2} (x/2)e^{\frac{|x|^2}{4}}v \rangle_{L^2(\mu)} \nonumber\\
    & + & \langle k^{-1/2}e^{\frac{|x|^2}{4}} Dv, k^{-1/2}e^{\frac{|x|^2}{4}} Dv \rangle_{L^2(\mu)} \nonumber\\
    & + & \langle  k^{-1/2}(x/2)e^{\frac{|x|^2}{4}}v, k^{-1/2}e^{\frac{|x|^2}{4}}Dv \rangle_{L^2(\mu)}\nonumber\\
    & + & \langle k^{-1/2}e^{\frac{|x|^2}{4}}Dv, k^{-1/2}(x/2)e^{\frac{|x|^2}{4}}v \rangle_{L^2(\mu)}\nonumber\\
    & = &  \int_{\R^n} |Dv(x)|^2\,dx+\int_{\R^n} \frac{|x|^2}{4}|v(x)|^2\,dx+I_1+I_2\,,
    \end{eqnarray}
    where
    \[
    I_1:=\langle  k^{-1/2}(x/2)e^{\frac{|x|^2}{4}}v, k^{-1/2}e^{\frac{|x|^2}{4}}Dv \rangle_{L^2(\mu)}= \langle  (x/2)v, Dv \rangle_{L^2(\R^n)}\,,
    \]
    and 
    \[
    I_2:=\langle k^{-1/2}e^{\frac{|x|^2}{4}}Dv, k^{-1/2}(x/2)e^{\frac{|x|^2}{4}}v \rangle_{L^2(\mu)}=\langle Dv, (x/2)v \rangle_{L^2(\R^n)}\,.
    \]
   Now, to calculate the sum $I_1+I_2$, recall the following fact by the general theory 
   \[
   2 {\rm Re} ( \langle  xv, Dv \rangle_{L^2(\R^n)})= \langle  xv, Dv \rangle_{L^2(\R^n)}+\langle Dv, xv \rangle_{L^2(\R^n)}\,.
   \]
  Observe that integrating by parts we get
   \[
   \sum_{j, A} \int_{\R^n} x_j v_A (\partial_j v_A) \,dx=-\sum_{j, A} \int_{\R^n} v_{A}^{2}\,dx-\sum_{j,A} \int_{\R^n} x_j (\partial_j v_A)v_A\,dx\,.
   \]
   Therefore, since $\|v\|^{2}_{L^2}=\sum_A \int_{\R^n}v_{A}^{2}dx=1$, we have 
   \begin{equation}
       \label{EQ:gauss.6}
        \sum_{j, A} \int_{\R^n} x_j v_A (\partial_j v_A) \,dx=-\frac{1}{2}\sum_{j,A}\int_{\R^n}v_{A}^{2}\,dx=-\frac{1}{2}\sum_{j,A}1=-\frac{n}{2}\,.
   \end{equation}
   We have
   \begin{eqnarray*}
      {\rm Re} ( \langle  xv, Dv \rangle_{L^2(\R^n)}) & = & {\rm Re} \left( \int_{\R^n} \sum_{j, A} \overline{e_je_A}x_j v_A  \sum_{k,B}e_ke_B \partial_k v_B\,dx\right) \\
      & = &  {\rm Re} \left( \int_{\R^n} \sum_{j, A,k,B} \overline{e_je_A}e_ke_Bx_j v_A  \partial_k v_B\,dx\right)\\
      & = & \sum_{j=k,A=B}\int_{\R^n} x_j v_A \partial_k v_B\,dx\\
      & = & \sum_{j,A} \int_{\R^n} x_j v_A \partial_j v_A\,dx\\
      & \stackrel{\eqref{EQ:gauss.6}}= & -\frac{n}{2}\,,
   \end{eqnarray*}
   which, by the above, implies that
   \[
   \langle  xv, Dv \rangle_{L^2}+\langle Dv, xv \rangle_{L^2}=-n\,,
   \]
   and it follows that 
   \begin{equation}\label{EQ:gauss.3}
        I_1+I_2=-\frac{n}{2}\,.
   \end{equation}
  Summarising, by \eqref{EQ:gauss.2} and \eqref{EQ:gauss.3} we get that 
  \begin{equation}
      \label{EQ:gauss.4}
       \langle Du, Du \rangle_{L^2(\mu)}=\int_{\R^n} |Dv(x)|^2\,dx+\int_{\R^n} \frac{|x|^2}{4}|v(x)|^2\,dx-\frac{n}{2}\,.
  \end{equation}
  A combination of \eqref{EQ:gauss.1} and \eqref{EQ:gauss.4}, yields that the desired estimate \eqref{EQ:gauss(thm)}, will follow once we show that 
  \[
  \frac{n}{4}\log \left(\int_{\R^n} C_{1}^2 |Dv(x)|^2\,dx\right)+\log(k^{-1/2})\leq \int_{\R^n}  |Dv(x)|^2\,dx-\frac{n}{2}\,,
  \]
  or that 
  \[
   \frac{n}{4}\log \left(\int_{\R^n} C_{1}^2 |Dv(x)|^2\,dx\right)+\log(k^{-1/2})+\log \left(e^{n/2} \right)\leq \int_{\R^n}  |Dv(x)|^2\,,
  \]
  and since by the properties of the logarithm
  \begin{equation*}
      \begin{split}
          \frac{n}{4}\log \left(\int_{\R^n} C_{1}^2 |Dv(x)|^2\,dx\right)&+\log(k^{-1/2})+\log \left(e^{n/2} \right) \\&
           =\frac{n}{4}\left[\log \left(\int_{\R^n} C_{1}^2 |Dv(x)|^2 k^{-2/n}e^2\,dx\right) \right]\,,
      \end{split}
  \end{equation*}
  the latter inequality is equivalent to showing that 
   \begin{equation}\label{EQ:gauss.5}
       \log \left(\int_{\R^n} C_{1}^2 k^{-2/n}e^2 |Dv(x)|^2\,dx\right)\leq \frac{4}{n}\int_{\R^n}  |Dv(x)|^2\,dx\,.
   \end{equation}
    Using the fact that for all $r>0$, we have $\log r\leq r-1$, we can see that \eqref{EQ:gauss.5} holds true if 
    \[
    e^{-1}C_{1}^2 k^{-2/n}e^{2}\int_{\R^n} |Dv(x)|^2\,dx \leq \frac{4}{n} \int_{\R^n} |Dv(x)|^2\,dx\,.
    \]
   This is true if $eC_{1}^2 k^{-2/n}=\frac{4}{n}$, that is, for $k=\left(\frac{C_1}{2}\sqrt{ne}\right)^{n}$, completing the proof. 
\end{proof}
Recall the following result that appeared e.g. in \cite[Section 1]{BCLSC}.
\begin{prop}\label{PROP:Nash.Sobolev.s}
    If $0< f \in C^{\infty}(\mathbb{R}^n)$ is a positive (real valued) and smooth function that satisfies  Nash inequality 
    \begin{equation}
        \label{EQ:Nash.s}
        \|f\|_{L^2(\R^n)}^{1+2/n}\leq C \|\nabla f\|_{L^2(\R^n)} \| f\|_{L^1(\R^n)}^{2/n}\,,
    \end{equation}
    then $f$ satisfies Sobolev inequality 
    \begin{eqnarray}
        \label{EQ:Sobolev.s}
        \|f\|_{L^{\frac{2n}{n-2}}(\R^n)}\leq C \|\nabla f\|_{L^2(\R^n)}\,,
    \end{eqnarray}
    where $\nabla f$ stands for the gradient of $f$, where the constant $C$ as in \eqref{EQ:Nash.s} and \eqref{EQ:Sobolev.s} is in principle not the same.
\end{prop}
\begin{rem}\label{REM:Nash.Sobolev}
    One can show that if Nash inequality 
    \begin{equation}\label{EQ:NaschC}
        \|f\|_{L^2(\R^n,\R_{0,m})}^{1+2/n}\leq C \|D f\|_{L^2(\R^n,\R_{0,m})} \| f\|_{L^1(\R^n,\R_{0,m})}^{2/n}\,,
    \end{equation}
    is satisfied for all Clifford-valued $f$, then the Sobolev inequality 
   \begin{equation}\label{EQ:SobolevC}
       \|f\|_{L^{\frac{2n}{n-2}}(\R^n,\R_{0,n})}\leq C' \|D f\|_{L^2(\R^n,\R_{0,n})}\,, 
   \end{equation}
    is satisfied for all Clifford-valued $f$.  Indeed, observe that Proposition \ref{PROP:Nash.Sobolev.s} can be extended to all real-valued $f \in C^{\infty}(\mathbb{R}^n)$ in view of the equality $|\nabla |f||=|\nabla f|$. Now suppose that inequality \eqref{EQ:NaschC} is satisfied for all Clifford-valued $f$. Then using Proposition \ref{PROP:Nash.Sobolev.s} we can estimate 
    \begin{eqnarray*}
        \|f\|_{L^{\frac{2n}{n-2}}(\R^n,\R_{0,m})} & \lesssim & \sum_{A} \|f_A\|_{L^{\frac{2n}{n-2}}(\R^n,\R_{0,m})}\\
        & \lesssim & \sum_{A} \| \nabla f_A\|_{L^2(\R^n,\R_{0,m})}\\
        & = & \sum_{A} \| D f_A\|_{L^2(\R^n,\R_{0,m})}\\
        & \lesssim & \|D f\|_{L^2(\R^n,\R_{0,m})}\,,
    \end{eqnarray*}
    where we have used the fact that for the real-valued function $f_A$ we have $Df_A=\nabla f_A$, and the  notation $ a\lesssim b$ to denote that there exists some constant $C \in \mathbb{R}$ for which we have $a \leq Cb$.
\end{rem}

Finally, we derive the generalised Poincar\'e inequality. 

\begin{thm}\label{THM:poincare}
Let $m\geq n\geq 3$.
For $v \in L^2(\R^n, d\mu, \R_{0,m})$ such that $\|v\|_{L^2(\mu,\R_{0,m})}=1$ we have
\begin{equation}
    \label{EQ:poincare1(thm)}
    \int_{\R^n} |v|^2\,d\mu-\left(\int_{\R^n} |v|^{2/q}\,d\mu \right)^q\leq 2(q-1) \int_{\R^n} |Dv|^2\,d\mu\,,
\end{equation}
for all $q\geq 1$, where the Gaussian measure $\mu$ is an in Theorem \ref{THM:gauss}. 
\end{thm}
The above theorem is still valid in the case of  $v \in L^2(\R^n, d\mu, \mathbb{C}_{m})$.

\begin{proof}
    For $v$ as in the hypothesis, we consider the function $m: (1, \infty) \rightarrow \R$ defined by \[m=m(q):= q \log (\int_{\R^n} |v|^{2/q}\,d\mu):=q \log I(q)\,,\]
    where 
    \[
    I(q)=\int_{\R^n} |v|^{2/q}\,d\mu\,.
    \]
    Then 
    \[
      I'(q)=\int_{\R^n} \log |v||v|^{\frac{2}{q}}\left(-\frac{2}{q^2}\right)\,d\mu\,,
    \]
    and 
    \[
    I''(q)=\int_{\R^n} \left((\log|v|)^{2}|v|^{\frac{2}{q}}\left(-\frac{2}{q^2}\right)^{2}+\log |v||v|^{\frac{2}{q}}\left(\frac{4}{q^3}\right)\right)\,d\mu\,.
    \]
    Moreover since 
     \[
   m'(q)=\log I(q)+\frac{q I'(q)}{I(q)}\,,
   \]
   we get that 
   \begin{equation}
   \begin{split}
       m''(q)&=\frac{I(q)(2I'(q)+qI''(q))-q(I'(q))^{2}}{I^{2}(q)}\\&
       =\frac{q}{I^{2}(q)}\Biggl{[}\left(\int_{\R^n}|v|^{\frac{2}{q}}d\mu\right)\left(\int_{\R^n} (\log|v|)^{2}|v|^{\frac{2}{q}}\left(-\frac{2}{q^2}\right)^{2}d\mu\right)\\&
       -\left(\int_{\R^n} \log |v||v|^{\frac{2}{q}}\left(-\frac{2}{q^2}\right)\,d\mu\right)^{2}\Biggl{]}.
   \end{split}
   \end{equation}
   The latter shows that $m$ is a convex function since by the Cauchy-Schwarz inequality we have 
   \[
   \left( \int_{\R^n} |v|^{\frac{2}{q}}\,d\mu\right)\left(\int_{\R^n} (\log |v|)^2 |v|^{\frac{2}{q}}\left(-\frac{2}{q^2}\right)^2\,d\mu \right) \geq \left(\int_{\R^n} \log |v||v|^{\frac{2}{q}}\left(-\frac{2}{q^2}\right)\,d\mu \right)^2\,.
   \]
   Therefore the function $q \mapsto e^{m(q)}$ is also convex, and so the function $\psi:[1,\infty) \rightarrow \R$ defined via
    \[
   \psi(q):= \frac{e^{m(1)}-e^{m(q)}}{1-q}=\frac{\int_{\R^n} |v|^2\,d\mu-\left( \int_{\R^n} |v|^{\frac{2}{q}}\,d\mu\right)^q}{1-q}\,,
   \]
   is monotonically non-decreasing, and we have $\psi(q)\geq \lim_{q \rightarrow 1+}\psi(q)$.  
   We have 
   \begin{eqnarray*}
       \frac{d}{dq}\left(\int_{\R^n} |v|^{\frac{2}{q}}\,d\mu \right)^q & = &
       \left(\int_{\R^n} |v|^{\frac{2}{q}}\,d\mu \right)^q \frac{d}{dq} \left[q \log \left(\int_{\R^n} |v|^{2/q}\,d\mu \right) \right] \\
      & = &  \left(\int_{\R^n} |v|^{\frac{2}{q}}\,d\mu \right)^q \left[\log \left(\int_{\R^n} |v|^{2/q}\,d\mu \right)-\frac{2}{q}\frac{\int_{\R^n} \log |v||v|^{2/q}d\mu}{\int_{\R^n} |v|^{2/q}d\mu}  \right]\,.
   \end{eqnarray*}
   Hence using the l'H\^{o}pital rule we can compute 
   \begin{eqnarray*}
     \lim_{q \rightarrow 1}\frac{d}{dq}\left(\int_{\R^n} |v|^{\frac{2}{q}}\,d\mu \right)^q & = & \int_{\R^n} |v|^2d\mu \cdot \log \left(\int_{\R^n} |v|^2d\mu \right)-2\int_{\R^n} \ |v|^2 \log|v|\,d\mu\\
     & = & -2\int_{\R^n} \ |v|^2 \log|v|\,d\mu\,,
   \end{eqnarray*}
   since $\|v\|_{L^2(\mu)}=1$, and we have proved that  $\lim\limits_{q \rightarrow 1}\psi(q)=-2 \int_{\R^n} \ |v|^2 \log|v|\,d\mu$.
   Summarising the above we get 
   \[
    \int_{\R^n} |v|^2\,d\mu-\left( \int_{\R^n} |v|^{\frac{2}{q}}\,d\mu\right)^q \leq 2(q-1) \int_{\R^n} |v|^2 \log |v|\,d\mu\,,
   \]
    and inequality \eqref{EQ:poincare1(thm)} follows by an application of the Gaussian logarithmic Sobolev inequality \eqref{EQ:gauss(thm)} to estimate the right-hand side of the latter inequality.
    \end{proof}

\section{Zero modes}\label{SEC:zero}
The aim of this section is to establish estimates for the zero modes of the equation 
\begin{equation}\label{star}
    (D+Q)\psi=0\,, \quad {\rm in}\, \quad\B\,
\end{equation}
where $Q\in L^n(\mathbb{R}^n, \mathbb{C}_{n+1})$. While in~\cite{BES08} the authors consider $Q\in L^3(\R^3, \mathbb{C}^{4\times 4})$ the physically important situation is the case of $Q\in L^3(\mathbb{R}^3, \mathbb{C}_4)\subset L^3(\mathbb{R}^3,\mathbb{C}^{4\times 4})$ (see, e.g.,~\cite{Hile90}), the embeddding being given by the Dirac matrices.
We will see that the obtained results not only generalise and improve the ones obtained in \cite{BES08}, but also simplify their proofs. Among other things, we will not require any regularity of the potential $Q$, that was required in \cite{BES08} for their method to work. We also obtain better decay rates of zero modes.

\begin{thm}\label{THM:zero}
    Suppose $n\geq 3$ and  let $\psi \in L^2(\R^n, \mathbb{C}_{n+1})$ satisfy the equation
    \begin{equation}\label{EQ:zeromodes}
    D\psi=-Q\psi\,\quad {\rm in} \quad \R^n\,,
    \end{equation}
    for some potential $Q=Q(x)$ such that $Q(x)=O(\langle x\rangle ^{-1})$.  Then we have 
    \[
    \||x|\psi\|_{L^{\frac{2n}{n-2}}(\R^n,\mathbb{C}_{n+1})}<\infty\,.
    \]
\end{thm}
\begin{proof}
By Sobolev inequality (Theorem \ref{THM:Sobolev2}) we have
    \begin{eqnarray*}
         \||x|\psi\|_{L^{\frac{2n}{n-2}}} & \leq & C_{1} \|D(|x|\psi)\|_{L^2}\\
         & \leq & C_{1} (\||x|D\psi\|_{L^2}+\|\psi\|_{L^2})\\
         & = & C_{1} (\||x|Q\psi\|_{L^2}+\|\psi\|_{L^2})\\
          & \lesssim & C_{1} (\||x|\langle x \rangle^{-1}\psi\|_{L^2}+\|\psi\|_{L^2})<\infty\,,
    \end{eqnarray*}
    where $C_1$ is as in the Remark \ref{REM:C1}, and the result follows.
\end{proof}
\begin{cor}\label{COR:zero1}
Let $\B$ be the complement of the unit ball $B_1 \subset \R^3$ and  $\psi \in L^2(\B, \mathbb{C}_4)$ be the solution to the equation \eqref{EQ:zeromodes}, for $Q$ such that $Q(x)=O(\langle x\rangle ^{-1})$ in $\B$. Then for any $0<k<6$, and $\alpha>3+k/2$, the integral 
\begin{equation}\label{EQ:cor.zero,1}
     \int_{B_{1}^{c}} |\psi(x)|^{k}|x|^{2k-\alpha}\,dx<\infty\,,
\end{equation}
   converges.
\end{cor}
\begin{proof}
    We have
    \begin{eqnarray*}
        \int_{B_{1}^{c}} |\psi(x)|^{k}|x|^{2k-\alpha}\,dx & = &  \int_{B_{1}^{c}} |\psi^{k}(x)||x|^{k}|x|^{k-\alpha}\,dx\\
        & \leq & \left(\int_{\B} |\psi(x)|^{pk}|x|^{pk}\,dx\right)^{1/p} \left(\int_{\B}|x|^{(k-\alpha)\frac{p}{p-1}} \,dx\right)^{\frac{p-1}{p}}\\
        & = & \||\psi||x|\|_{L^{pk}}^{k}\left(\int_{\B}|x|^{(k-\alpha)\frac{p}{p-1}} \,dx\right)^{\frac{p-1}{p}}
    \end{eqnarray*}
    where we have applied H\"older's inequality for $p> 1$ and used that $\frac{1}{p}+\frac{p-1}{p}=1$. Now observe that by Theorem \ref{THM:zero} for $n=3$ the norm $\||\psi||x|\|_{L^{kp}}$ is finite for $kp=6$. Notice also that for $\alpha>3+k/2$ and for $p=6/k$, with $k \in (0,6)$ we have 
    \[
    k-\alpha<k/2-3<0 \quad {\rm and}\quad \frac{p}{p-1}=\frac{6}{6-k}>0
    \]
    which implies that
    \[
    (k-\alpha)\frac{p}{p-1}<(k/2-3)\left( \frac{6}{6-k}\right)=-3\,.
    \]
    By changing to polar coordinates  the integral $\int_{\R^3}|x|^{-\rho}\,dx<\infty$, for $\rho>0$, is finite if $\rho>3$. Hence by the above the integral  
    \[
    \int_{\B}|x|^{(k-\alpha)\frac{p}{p-1}} \,dx
    \]
    is finite, and the proof is complete.
\end{proof}

\begin{rem}\label{REM:comp}
    Corollary \ref{COR:zero1} generalises \cite[Theorem 3]{BES08}, since the latter is a special case of Corollary \ref{COR:zero1} for the range $k \in [1,10/3)$ and for $\alpha=6$. Moreover, let us point out that in Corollary \ref{COR:zero1} we also do not require the regularity assumption for the potential $Q$, i.e., we only assume that $Q \in L^3(\R^3, \mathbb{C}_4)$, compared to additional regularity required in \cite[Theorem 3]{BES08}. 
\end{rem}

We will now show that Corollary \ref{COR:zero1} actually improves also the following theorem as in \cite[Theorem 4]{BES08}.
\begin{thm}\label{THM.BSE}
Suppose that $\psi \in L^2(\B,\mathbb{C}_4)$ is the solution to the equation \eqref{star}, for some $Q$ such that $Q\in L^3(\R^3, \mathbb{C}_4)$ and $Q(x)=O(|x|^{-1})$ in $\B$. Then we have 
\begin{equation}\label{EQ:rem.zero}
    \int_{\B}|\psi(x)|^{s} |x|^{(2+t)s-6}\,dx<\infty\,,
\end{equation}
    for any $0<t<11/10$ and $s\in [1,4/3)$.
\end{thm}

\begin{rem}
    To compare the result of Theorem \ref{THM.BSE} with Corollary  \ref{COR:zero1}, we first notice that 
    \[
      \int_{\B}|\psi(x)|^{s} |x|^{(2+t)s-6}\,dx = \int_{B_{1}^{c}} |\psi(x)|^k|x|^{2k-\alpha}\,dx\,,
    \]
     where we have set $k=s$ and $\alpha=-tk+6$. By writing the integral \eqref{EQ:rem.zero} in this form we see that by Theorem \ref{THM.BSE} the smallest $\alpha$ that allows for convergence of \eqref{EQ:rem.zero} is $\alpha=6-\frac{11}{10}k$. On the other hand by Corollary \ref{COR:zero1} the smallest such $\alpha$ is equal to $3+k/2$. Now, comparing $6-\frac{11}{10}k$ with $3+k/2$, we can easily verify that for all $k\in [1,4/3)$ we have 
     \[
     6-\frac{11}{10}k>3+k/2\,,
     \]
     which in turn implies that Corollary \ref{COR:zero1} improves the result of Theorem \ref{THM.BSE} that appeared in \cite{BES08}. Moreover, as pointed out in Remark \ref{REM:comp} where Corollary \ref{COR:zero1} is compared with \cite[Theorem 3]{BES08}, the assumptions on the regularity of the potential $Q$ restrict even more the conditions under which we have the convergence of \eqref{EQ:rem.zero}, and such assumptions are not required in Corollary \ref{COR:zero1}.
\end{rem}
\section*{Acknowledgment} There is no conflict of interests. The manuscript has no associated data.

\end{document}